\documentclass[12pt]{amsart}

\usepackage{latexsym, amssymb}

\newcommand{\be}{\begin{equation}}
\newcommand{\ee}{\end{equation}}
\newcommand{\beq}{\begin{eqnarray}}
\newcommand{\eeq}{\end{eqnarray}}

\newtheorem{thm}{Theorem}[section]

\newtheorem{lma}{Lemma}[section]
\newtheorem{prop}{Proposition}[section]
\newtheorem{cor}{Corollary}[section]
\theoremstyle{definition}
\newtheorem{defn}{Definition}[section]
\theoremstyle{remark}
\newtheorem{rem}{Remark}[section]
\numberwithin{equation}{section}

\def\E{\mathcal{E}}
\def\p{\partial}
\def\S{\Sigma}

\def\R{\mathbb{R}}
\def\H{\mathbb{H}}

\def\Sp{\Sigma^\prime}

\def\tr{{\rm tr}}

\def\p{\partial}
\def\lf{\left}
\def\ri{\right}

\def\ol{\overline}
\def\R{\Bbb R}

\def\la{\langle}
\def\ra{\rangle}

\def\Ric{\text{\rm Ric}}
\def\div{\text{\rm div}}

\def\Pi{\displaystyle{\mathbb{II}}}

\def\E{\mathcal{E}}

\def\Sp{\mathcal{S}}
\def\X{\mathcal{X}}
\def\M{\mathcal{M}}

\def\stop{\hfill$\Box$}

\begin{document}

\title{On geometric problems related to Brown-York and Liu-Yau quasilocal mass}

\begin{abstract}
We discuss some geometric problems related to the definitions of quasilocal mass
proposed by  Brown-York \cite{BYmass1} \cite{BYmass2} and Liu-Yau \cite{LY1} \cite{LY2}.
Our discussion consists of three parts. In the first part, we propose a new
variational problem on compact manifolds with boundary, which is motivated by
the study of Brown-York mass. We prove that critical points of
this variation problem are exactly static metrics. In the second part, we derive
a derivative formula for the Brown-York mass of a smooth family of closed 2 dimensional surfaces
evolving in an ambient three dimensional manifold. As an interesting by-product, we are able to
write the ADM mass \cite{ADM61} of an asymptotically flat $ 3$-manifold  as the sum of
the Brown-York mass of a coordinate sphere $ S_r $ and an integral of the scalar curvature
plus a geometrically constructed function $ \Phi (x) $ in the asymptotic region  outside $ S_r $.
In the third part, we prove that for any closed,  spacelike,  $ 2$-surface $ \Sigma $
in the Minkowski space $ \R^{3,1} $
for which the Liu-Yau mass is defined, if $ \Sigma $ bounds a compact spacelike hypersurface
in $ \R^{3,1}$,  then the Liu-Yau mass of $ \Sigma $ is  strictly positive unless $ \Sigma $ lies on a hyperplane.
We also show that the examples given by \'{O}  Murchadha,   Szabados  and Tod  \cite{MST} are special cases of this result.
\end{abstract}

\keywords{quasilocal mass, static metric, scalar curvature}

\renewcommand{\subjclassname}{\textup{2000} Mathematics Subject Classification}
 \subjclass[2000]{Primary 53C20; Secondary 58JXX}
\author{Pengzi Miao$^1$, Yuguang Shi$^2$
and Luen-Fai Tam$^3$}

\address{School of Mathematical Sciences, Monash University,
Victoria, 3800, Australia.} \email{Pengzi.Miao@sci.monash.edu.au}

\address{Key Laboratory of Pure and Applied mathematics, School of Mathematics Science, Peking University,
Beijing, 100871, P.R. China.} \email{ygshi@math.pku.edu.cn}

\address{The Institute of Mathematical Sciences and Department of
 Mathematics, The Chinese University of Hong Kong,
Shatin, Hong Kong, China.} \email{lftam@math.cuhk.edu.hk}

\thanks{$^1$ Research partially supported by Australian Research Council Discovery Grant  \#DP0987650}

\thanks{$^2$Research partially supported by NSF grant of China and
Fok YingTong Education Foundation.}
\thanks{$^3$Research partially supported by Hong Kong RGC General Research Fund  \#GRF 2160357}
\date{June 2009}

\maketitle
 \markboth{Pengzi Miao, Yuguang Shi  and Luen-Fai Tam}
 {Geometric problems related to Brown-York and Liu-Yau quasilocal mass}

\section{Introduction}
In this work, we will discuss some geometric problems related to the definitions of quasilocal mass
proposed by  Brown-York \cite{BYmass1} \cite{BYmass2} and Liu-Yau \cite{LY1} \cite{LY2}.
In general, there are certain properties that a reasonable definition of quasilocal mass should satisfy, see \cite{ChristodoulouYau1986} for example. The most important property is the positivity. There are results on positivity of Brown-York mass and Liu-Yau mass in \cite{ShiTam02,LY2,WangYau07,ShiTam07,WangYau08}.
In particular, the following is a consequence on the positivity of Brown-York mass proved by the last two authors in \cite{ShiTam02}.
Let $ g_e $ be the standard Euclidean metric on $ \R^3 $.
Let $\Omega$ be a bounded strictly convex domain in $\R^3$  with smooth boundary $\Sigma$ which has  mean curvature $H_0$.
Then  $\int_\Sigma H_0d\sigma$ is a maximum of the functional $\int_\Sigma Hd\sigma$ on the class of smooth metrics with nonnegative scalar curvature on $\Omega$ which agree with $ g_e $ tangentially on $ \Sigma $ and have positive boundary mean curvature $H$.  It is interesting to see if this is still true for general domains
in $ \R^3 $.

 In \cite{ShiTam07}, a similar result was proved for domains in $\H^3$, the hyperbolic $3$-space. Namely, it was proved that
 if $ g_h $ is the standard hyperbolic metric on $ \H^3 $ and
 $\Omega $  is a bounded domain with strictly convex smooth boundary $\Sigma$ which is a topological sphere and has mean
 curvature $ H_0 $, then $\int_\Sigma H_0\cosh r d\sigma$ is a maximum of the functional $\int_\Sigma H\cosh rd\sigma$ on the  class of smooth metrics with  scalar curvature bounded below
 by $ -6$  which agree with $ g_h $ tangentially  on $\Sigma$ and have  positive boundary mean curvature $H$. Here $r$ is the distance function on $\H^3$ from a fixed point in $ \Omega$.  Again it is interesting to see if this is still true for general domains in $ \H^3$.

The results and questions above motivate us to study the functional
$$
F_\phi(g)=\int_\Sigma H\phi \ d\sigma,
$$
where  $\Sigma$ is the boundary of an $n$ dimensional compact manifold $\Omega$,
$ \phi $ is a given   smooth nontrivial   function (that is $\phi\not\equiv0$) on $ \Sigma $,
and $ d \sigma $ is the volume form of a fixed metric $ \gamma $
on $ \Sigma$.  The class of metrics $g$ we are interested
is the space of metrics with constant scalar curvature $K$ which
induce the metric $\gamma$ on $\Sigma$.  In Theorem \ref{static-t1}, we will prove the following:
 {\it $g$ is a critical point of $F_\phi (\cdot ) $ if and only if $g$ is a static metric with a static
potential  $N$ that equals $ \phi $ on $ \Sigma $.  That is to say}:
$$
\left\{
\begin{array}{rcc}
- ( \Delta_gN )g + \nabla^2_g N -   N  \Ric(g) & = & 0, \ \ on \ \Omega \\
N & = & \phi, \ \ at \ \Sigma.
\end{array}
\right.
$$
In the theorem, for $K>0$, we also assume that the first Dirichlet eigenvalue of $(n-1)\Delta_g+K$ is positive.

In particular, if $\phi=1$, $K=0$  and $n=3$, we can conclude that
$g$ is a critical point of $ \int_\Sigma H \ d \sigma $ if
and only if $ g $ is a flat metric.

Another important question on quasilocal mass is whether
it has some monotonicity property.  In \cite{ShiTam02}, it was shown that the Brown-York mass of the  boundaries of certain domains in
a space  with some qausi-spherical metric is monotonically decreasing rather than increasing as the domains become larger.
In Theorem \ref{evolution-bymass}, we will derive {\it a more general formula for the derivative of the Brown-York mass of a smooth family of surfaces with positive Gaussian curvature which evolve in
an ambient manifold}. The formula gives a generalization of the
monotonicity formula in \cite{ShiTam02} which plays a key role
 in the proof  of the positivity of Brown-York mass. As an interesting
 by-product of this derivative formula,
 in Corollary \ref{ADMmassformula} we are able to write the ADM mass \cite{ADM61}  of an asymptotically flat $ 3$-manifold
 as the sum of the Brown-York mass of a coordinate sphere $ S_r $ and an integral of the scalar curvature plus a geometrically constructed function $ \Phi (x) $ in the asymptotic region  outside $ S_r $.

The Minkowski space $ \R^{3,1}$ represents the zero energy state
in general relativity. Thus, a reasonable notion of quasilocal mass should be such that
its value of a spacelike $2$-surface in $ \R^{3,1}$
equals zero.  In \cite{LY1,LY2}, the Liu-Yau mass was introduced and its positivity was proved. In the time symmetric case, this coincides with the Brown-York mass. However,  \'{O}. Murchadha,   Szabados  and   Tod  \cite{MST} constructed spacelike $2$-surfaces $\Sigma$ with spacelike mean curvature
vector $\vec H$ in  $ \R^{3,1}$ and with positive Gaussian curvature such that the Liu-Yau mass of $\Sigma$ given by
$$
\mathfrak{m}_{_{\rm LY}} (\Sigma) =\frac1{8\pi}\int_\Sigma (H_0-|\vec H|) d\sigma
$$
is strictly positive.  Here $H_0$ is the mean curvature of $\Sigma$ when  isometrically embedded in $\R^3$ and $ | \vec H | $
is the Lorentzian norm of $ \vec H $ in $ \R^{3,1}$.
Recently, Wang and Yau \cite{WangYau07,WangYau08} introduce   another definition of quasilocal mass to address this question. In Theorem \ref{LiuYau-t1} in this paper, we will prove the following:  {\it Let $\Sigma$ be a closed, connected,
spacelike $ 2$-surface in the Minkowski space $ \R^{3,1}$  with spacelike mean curvature vector and with positive Gaussian curvature. Suppose $\Sigma$ spans a compact, spacelike hypersurface in $ \R^{3,1}$, then the Liu-Yau mass of $\Sigma$ is strictly positive, unless $ \Sigma $ lies on a hyperplane}. The results give some properties on isometric embeddings of  compact   surfaces with positive Gaussian curvature in the Minkowski space.
We will also show that all the examples in \cite{MST} satisfy the conditions in Theorem \ref{LiuYau-t1}.

This paper is organized as follows. In  section 2, we will prove that static metrics are the only critical points of the functional
$F_\phi(\cdot) $. In  section 3, a formula for the derivative of the Brown-York mass will be derived and some applications will be  given. In  section 4, we will prove that for  ``most" spacelike $ 2$-surfaces  in $ \R^{3,1} $ for which the Liu-Yau mass is defined, their Liu-Yau mass is strictly positive. In the appendix, we prove some
results on the differentiability of a $1$-parameter family of isometric
embeddings in $ \R^3$, following the arguments of
Nirenberg  \cite{Nirenberg}. The results will be used in section 3.

We would like to thank Robert Bartnik for useful discussions on  the existence of maximal surfaces.

\section{static metrics and Brown-York type integral}

Throughout this section, we let $\Omega$ be an $ n $-dimensional ($ n \ge 3$) compact manifold with smooth boundary $\Sigma$.
Let $\gamma$ be a smooth Riemannian  metric on $\Sigma$. As in \cite{MiaoTam08}, for a constant $K$ and any integer $ k > \frac{n}{2} + 2 $,
we let $\mathcal{M}_{\gamma}^K$ be the set of $W^{k,2}$ metrics $g$ on $\Omega$ with constant scalar curvature $K$ such that $g|_{T(\Sigma)}=\gamma$.
If $g\in \mathcal{M}_{\gamma}^K$ and the first Dirichlet eigenvalue of $(n-1)\Delta_g+K$ is positive,
 where $ \Delta_g $ is the usual Laplacian operator of $ g $, then $\mathcal{M}_{\gamma}^K$ is a manifold near $g$ (see \cite{MiaoTam08} for detail).
 Let $ \phi $ be a given smooth function on $ \Sigma $,  we define the following functional on $\mathcal{M}_{\gamma}^K$:
 \be  \label{weightedBY-e1}
 F_\phi(g)=\int_\Sigma H_g\phi \ d\sigma ,
 \ee
where $H_g$ is the mean curvature of $\Sigma$ in $ (\Omega, g) $
 with respect to the outward unit normal and $ d \sigma $ is the volume form of $ \gamma $.
 Motivated by the results in \cite{ShiTam02,ShiTam07} on the positivity of
Brown-York mass and some generalization, we want to determine the critical points of $ F_\phi ( \cdot ) $ on $ \mathcal{M}_\gamma^K $.

Before we state the main result, we recall the following definition from \cite{Corvino00}:
\begin{defn}
A metric $ g $ on an open set $ U $  is called a {\it static} metric on $ U $ if
 there exists a nontrivial function $ N $ (called the {\it static potential})  on $ U $ such that
 \be \label{staticeq}
 - ( \Delta_gN )g + \nabla^2_g N -   N  \Ric(g)  = 0 .
 \ee
 Here $ \Delta_g $, $ \nabla^2_g $ are the usual Laplacian, Hessian operator of $ g $ and
 $ \Ric (g) $ is the Ricci curvature of $ g $.
\end{defn}

A basic property of static metrics is that they are necessarily metrics of constant scalar curvature \cite[Proposition 2.3]{Corvino00}.

In the following, we obtain a characterization of static metrics in $ \mathcal{M}^K_\gamma $ using the function $ F_\phi ( \cdot ) $.

\begin{thm}\label{static-t1}
With the above notations, let $ \phi $ be a nontrivial smooth   function on $ \Sigma $.
Suppose $ g \in \mathcal{M}^K_\gamma $ such that the first Dirichlet eigenvalue of
$ ( n - 1 ) \Delta_g + K $ is positive.
Then $ g $ is a critical point of $ F_\phi ( \cdot ) $ defined in \eqref{weightedBY-e1}
 if and only if
$ g $ is a static metric  with a static potential $ N $ such that
$ N = \phi $ on $ \Sigma $.
\end{thm}

\begin{proof} Since the Dirichlet eigenvalue of  $ ( n - 1 ) \Delta_g + K $ is positive, we know
 $\mathcal{M}^K_\gamma $ is a manifold near $g$ by the result in \cite{MiaoTam08}.

First, we suppose $g $ is a static metric with a potential $N$ such that $N=\phi$ on $\Sigma$.
Let $g(t)$ be a smooth curve in $\mathcal{M}^K_\gamma$ with $g(0)=g$. Let
$$
F(t) = \int_\Sigma H(t) \phi \ d\sigma,
$$
where $H(t)$ is the mean curvature of $ \Sigma $ in $ (\Omega, g(t) ) $ with respect to the
outward unit normal $ \nu $. Let $`` \ ' \ " $ denote the derivative with respect to $ t $.
We want to prove that
$F'(0)=0$. Let $h=g'(0)$. In what follows, we let $\omega_n$ denote the outward unit normal part
of a $1$-form $\omega$, i.e. $ \omega_n = \omega ( \nu   ) $, let $\mathbb{II}$ be the
second fundamental form of $\Sigma$ in $ (\Omega, g(t) ) $ with respect to $ \nu $,
let $ X $ be the vector field on $ \Sigma $ that is dual to the $1$-form $ h(\nu, \cdot)|_{T (\Sigma)}$
on $ (\Sigma, \gamma )$  and let $ \div_\gamma X $ be the divergence of $ X $ on $ (\Sigma, \gamma)$.
For convenience, we  often omit writing the volume form in an integral.
As in   \cite[(34)]{MiaoTam08}, we have
\begin{equation}\label{static-e1}
\begin{split}
2 F^\prime (0) & =   \int_\Sigma 2 H^\prime (0) N  \\
&=\int_\Sigma N\lf([  d ( \tr_g h )  - \div_g h ]_n - \div_\gamma X
- \la \mathbb{II}, h \ra_\gamma\ri)  \\
& =   \int_\Sigma N [ d ( \tr_g h) - \div_g h ]_n  -\int_\Sigma N\div_\gamma X  \ \  \text{(because $ h |_{ T \Sigma } = 0 $)} \\
&=  \int_\Omega N\lf[\triangle_g ( \tr_g h ) - \div_g ( \div_g h )\ri]-\int_\Omega \tr_g h\Delta N +\int_\Sigma \tr_g h(dN)_n\\
&-\int_\Omega \la dN,\div_g h\ra_g-\int_\Sigma N\div_\gamma X \\
&=   -\int_\Omega  N\la h, \Ric(g) \ra_g -\int_\Omega \tr_g h\Delta N +\int_\Sigma \tr_g h(dN)_n\\
& + \int_\Omega \la \nabla^2_g N, h\ra_g - \int_\Sigma h( \nu , \nabla N)
-\int_\Sigma N\div_\gamma X\ \  \text{(using \eqref{static-e2} and \eqref{static-e3})} \\
& =   \int_\Omega  \la h, - N \Ric(g) - ( \Delta_g N ) g + \nabla^2_g N  \ra_g \\
& +\int_\Sigma \tr_g h(dN)_n
- \int_\Sigma h( \nu , \nabla N)
-\int_\Sigma N\div_\gamma X ,
\end{split}
\end{equation}
where  we have used the facts
\begin{equation}\label{static-e2}
 \int_\Omega  \la dN,\div_g h\ra_g =
- \int_\Omega \la \nabla^2_g N, h \ra_g + \int_\Sigma h( \nu , \nabla N)
\end{equation}
and
\begin{equation}\label{static-e3} D R_g ( h) = \frac{d}{dt}R(t)|_{t=0}=-  \triangle_g ( \tr_g h ) + \div_g ( \div_g h )
- \la h, \Ric(g) \ra_g = 0 .
\end{equation}
Here $R(t)$ is the scalar curvature of $g(t)$.
Let $ \nabla^\Sigma N $ denote the gradient of $ N $ on $ (\Sigma, \sigma)$.
Integrating by parts on $ \Sigma $, we have
 \begin{equation}\label{static-e4}
 \int_\Sigma N\div_\gamma X=-\int_\Sigma\la \nabla^\Sigma N,X\ra_\gamma=-\int_\Sigma h( \nu ,\nabla N)+\int_\Sigma h( \nu ,\nu )(dN)_n .
 \end{equation}
 On the other hand, $ \tr_gh=h( \nu , \nu) $ at $ \Sigma $.
Hence, $ F^\prime (0) = 0 $ by \eqref{staticeq}, \eqref{static-e1}, and \eqref{static-e4}.

To prove the converse, suppose $g$ is a critical point of $F_\phi (\cdot ) $.
Let  $ \{ g(t) \} $ be any smooth path in  $\mathcal{M}^K_\gamma $
 passing $ g = g(0) $. Let $ h = g^\prime (0) $ and $ F( t) = F_\phi ( g(t) ) $. As before, we have
 \begin{equation} \label{wantstatic1}
\begin{split}
2 F^\prime (0) & =   \int_\Sigma 2 H^\prime (0) \phi  \\
& =   \int_\Sigma \phi [ d ( \tr_g h) - \div_g h ]_n  -\int_\Sigma \phi \div_\gamma X   .
\end{split}
\end{equation}
Since the first Dirichlet  eigenvalue of $( n - 1 ) \Delta_g   +  K$ is positive
and $\phi$ is not identically zero, there exists a unique smooth function  $ N = N_\phi $ which is not identically zero on $ \Omega $ such that

\begin{equation} \label{eqofN}
\left\{
\begin{array}{rcc}
( n - 1 ) \Delta_g N + K N & = & 0, \ \ on \ \Omega \\
N & = & \phi, \ \ at \ \Sigma.
\end{array}
\right.
\end{equation}
With such an $ N$ given, we have
\begin{equation} \label{usingN}
\begin{split}
&  \int_\Sigma \phi [ d ( \tr_g h) - \div_g h ]_n
-\int_\Sigma \phi \div_\gamma X \\
= & \int_\Omega N \lf[\triangle_g ( \tr_g h ) - \div_g ( \div_g h )\ri]-\int_\Omega \tr_g h\Delta N +\int_\Sigma \tr_g h(d N )_n\\
&-\int_\Omega \la d N,\div_g h\ra_g-\int_\Sigma N \div_\gamma X \\
= & \int_\Omega N (-1) \la h, \Ric(g) \ra_g  -\int_\Omega \tr_g h\Delta N  + \int_\Omega \la \nabla^2_g N, h\ra_g ,
\end{split}
\end{equation}
where we used the fact $ DR_g ( h ) = 0 $ (the boundary
terms canceled as before and  we have not used \eqref{eqofN} yet).

Now let $ \hat h $ be any smooth symmetric (0,2) tensor with   compact support in $ \Omega $.
For each $ t $ sufficiently small,  we can find a smooth positive function $u(t)$ on $ \Omega $
such that $u(t)=1$ at $ \Sigma $  and
$$  g (t) = u(t)^{\frac{4}{ n-2} } (g + t \hat h) \in \mathcal{M}^K_\gamma . $$
Moreover, $ u(t) $ is differentiable at $ t = 0 $ and $ u(0) \equiv 1 $ on $ \Omega $.
See the proof of \cite[Theorem 5]{MiaoTam08} for details on the existence of such a $ u(t) $.
Now $ g^\prime (0)  = \frac{4}{n-2} u'(0) g + \hat h. $ Hence, by \eqref{wantstatic1} and \eqref{usingN} we have
\begin{equation}
\begin{split}
2 F'(0) & = \int_\Omega
\la \frac{4}{n-2} u'(0) g + \hat h, - N \Ric(g) - ( \Delta_g N ) g + \nabla^2_g N \ra_g  \\
& =  \int_\Omega
\la  \hat h, - N \Ric(g) - ( \Delta_g N ) g + \nabla^2_g N \ra_g  \\
& + \int_\Omega \frac{4}{n-2} u'(0)
\lf[  - K N  - n ( \Delta_g N )  + \Delta_g N \ri] .
\end{split}
\end{equation}
By \eqref{eqofN}, the second integral in the above equation is zero. Hence, we have
\be
2 F'(0)   =  \int_\Omega
\la  \hat h, - N \Ric(g) - ( \Delta_g N ) g + \nabla^2_g N \ra_g  .
\ee
Since $ \hat h $ can be  arbitrary, we conclude that $ g $ and $ N $ satisfy \eqref{staticeq}.
\end{proof}

\begin{rem}
If $ K \leq 0 $, then the condition that the first Dirichlet eigenvalue of
$ ( n - 1 ) \Delta_g + K $ is positive  holds automatically for $ g \in \mathcal{M}^K_\gamma $.
\end{rem}

As a direct corollary of Theorem \ref{static-t1}, we have

\begin{cor}\label{static-c1}
With the notations given as in Theorem \ref{static-t1}, suppose $K=0$ and $\phi=1$.  Then $g \in \mathcal{M}^0_\gamma $ is a critical point of $ \int_\Sigma H \ d \sigma $  if and only if $g$ is a Ricci flat metric.
In particular, if $n=3$, then $g \in \mathcal{M}^0_\gamma $ is a critical point of $\int_\Sigma H \ d \sigma  $ if and only if  $g$ is a flat metric.
\end{cor}

If $ \phi $ does not change sign on the boundary, we further have:

\begin{cor}\label{static-c2}
With the notations given as in Theorem \ref{static-t1}, suppose $ \phi \ge 0 $ or $\phi\le 0$ on $ \Sigma $.
Suppose $g \in \mathcal{M}^K_\gamma$ is a static metric. If $ K > 0 $, we also
assume that the first Dirichlet eigenvalue of
$ ( n - 1 ) \Delta_g + K $ is positive.   Let $g(t)$ be a smooth family of smooth metrics on $ \Omega $ with $g(0)=g$ such that
\begin{enumerate}
\item[(i)] the scalar curvature of $g(t)$ is at least $K$.
\item[(ii)] $g(t)$ induces $\gamma$ on $\Sigma$.
\end{enumerate}
Then
$$
\frac{d}{dt}F_\phi(g(t))|_{t=0}=0.
$$
\end{cor}
\begin{proof} We prove the case that $\phi \ge0$ on $\Sigma$. The case that $\phi\le 0$ on $\Sigma$ is similar.

By the assumption of $g$, for $t$ small, we can find smooth positive functions $u(t)$ on $ \Omega$
with $u(t)=1$ on $\Sigma$ such that $\hat g(t)=u^{\frac{4}{n-2}}(t)g(t) \in \mathcal{M}^K_\gamma$,
 $ u(t) $ is differentiable at $ t = 0 $ and $ u(0) \equiv 1 $
(see the proof of Proposition 1 in \cite{MiaoTam08}). The mean curvature $\hat H(t)$ of $ \Sigma $
in $ (\Omega, \hat{g}(t) )$ is given by
\be
\hat H(t)=H(t)+\frac{ 2 ( n -1) }{n-2}\frac{\p u}{\p \nu_t },
\ee
where $H(t)$ and $ \nu_t $ are the mean curvature and the unit outward normal
of $ \Sigma $ in $(\Omega, g(t))$. Note that $u$ satisfies:
\begin{equation}\label{deformation-e1}
  \begin{cases}
     \frac{4(n-1)}{n-2} \Delta_{g(t)}u- K(t)u&=-Ku^{\frac{ n+2}{n-2}}, \text{\ in $\Omega$}\\
    u&=1,  \text{\ on $\Sigma$,}
\end{cases}
    \end{equation}
where  $K(t)$
is the scalar curvature of $g(t)$. Since $K(t)\ge K$, by the maximum principle, we have
$$
\frac{\p u}{\p \nu_t }\ge 0.
$$
Hence, $\hat H(t)\ge H(t)$ and consequently $F_\phi(\hat g(t))\ge F_\phi ( g(t) )$
by the assumption $ \phi \ge 0 $ on $ \Sigma$.
By Theorem \ref{static-t1}, we have
$$
\frac{d}{d t}F_\phi(\hat g(t))|_{t=0}=0.
$$
Since $ \hat g (0) = g (0) $,  we conclude
$$
\frac{d}{d t}F_\phi( g(t))|_{t=0}=0.
$$
\end{proof}

Here are some examples provided by Theorem \ref{static-t1}:

{\bf Example 1}:   Let $\Omega$ be a  bounded domain in $\R^n$ with smooth boundary $\Sigma$. Then the standard Euclidean metric is a critical point of $F_\phi(\cdot)$ with $\phi\equiv1$.  If $\Sigma$ is strictly convex, then this follows also from the result in \cite{ShiTam02}.

 {\bf Example 2}: Let $\Omega$ be a   bounded domain in $\H^n$ with smooth boundary $\Sigma$. Then the standard Hyperbolic metric is a critical point of $F_\phi(\cdot)$ with $\phi=\cosh r$, where $r$ is  the distance function on $ \H^n $ from a fixed point.  If $\Sigma$ is strictly convex and $n=3$, then this follows also from the result in \cite{ShiTam07}.

 {\bf Example 3}: Let $\Omega$ be a   domain in $\mathbb{S}^n$ with smooth boundary $\Sigma$ such that the volume of  $\Omega$  is less than  $2\pi$. Then the standard metric on $ \mathbb{S}^n $ is a critical point of $F_\phi(\cdot)$ with $\phi=\cos r$, where $r$ is  the distance function on $ \mathbb{S}^n $ from a  fixed point.

 {\bf Example 4}: Let $\Omega$ be a bounded domain with smooth boundary in the Schwarzschild manifold $\R^3\setminus\{0\}$ with metric
 $$
g_{ij}  =  \lf( 1+\frac{m}{2r} \ri)^4 \delta_{ij}
 $$
with $m>0$ and $r=|x|$. Then on $\Omega$  $ g $  is a critical point for $F_\phi(\cdot)$ with $\phi= (  1 - \frac m {2r} )/( 1 + \frac m {2r})$.

{\bf Example 5}: Complete conformally flat Riemannian manifolds with static metrics have been classified by Kobayashi in \cite[Theorem 3.1]{Kobayashi}.  In addition to the   manifolds in the previous examples, there are other kind of static metrics with $N$ being explicitly constructed.  Domains in these manifolds will be critical points of $F_\phi(\cdot)$ where $\phi$ is the restriction of $N$ to the boundary.  See \cite{Kobayashi} for more details.

\section{Derivative of the Brown-York mass}

In this section, we give a derivative formula that describes how the Brown-York mass of a surface
changes if the surface is evolving  in an ambient Riemannian manifold. Our main result is:

\begin{thm} \label{evolution-bymass}
Let $ \mathbb{S}^2 $ be the $2$-dimensional sphere.
Let $(M, g)$ be a $ 3 $-dimensional Riemannian manifold.
Let $ I $ be an open interval in $ \R^1$.
Suppose
$$
F : \mathbb{S}^2 \times I \longrightarrow M
$$
is a smooth map such that, for $ t \in I $,
\begin{enumerate}
\item[(i)] $ \Sigma_t = F(\mathbb{S}^2, t)$ is an embedded surface in $ M $ and $ \Sigma_t $ has positive Gaussian
curvature.
\item[(ii)] The velocity vector $ \frac{\p F }{\p t} $ is always perpendicular to $ \Sigma_t $, i.e
$$
\frac{ \p F}{\p t} = \eta \nu ,
$$
where $ \nu $ is a given unit vector field normal to $ \Sigma_t $ and $ \eta = \la \frac{\p F}{\p t} , \nu \ra$
denotes the speed of $ \S_t $ with respect to $ \nu $.
\end{enumerate}
Consider $ \mathfrak{m}_{_{\rm BY}} ( \Sigma_t ) $, the Brown-York mass of $ \Sigma_t $ in $(M, g)$, defined by
\be
\mathfrak{m}_{_{\rm BY}} (\Sigma_t) = \frac{1}{8\pi} \int_{\S_t} (H_0 - H) \ d \sigma_t,
\ee
where $H_0$ is the mean curvature of $\S_t $
with respect to the outward normal when isometrically embedded in $ \R^3 $, $ H $ is the mean curvature of $ \S_t $ with respect to $ \nu$ in $(M,g)$,
and $ d \sigma_t $ is the volume form of the induced metric on $ \S_t $.

We have
\be \label{formula-1}
\frac{ d  }{d t} \mathfrak{m}_{_{\rm BY}} ( \Sigma_t ) =
\frac{1}{ 16 \pi} \int_{\S_t} \lf( | A_0 - A |^2 - | H_0 - H |^2 + R \ri)  \eta \ d \sigma_t ,
\ee
where $ A_0 $ is  the second fundamental form of $ \S_t $ with respect to the outward normal when isometrically
embedded in $ \R^3 $, $ A $ is the second fundamental form of $ \S_t $ with respect to $ \nu $ in $ (M, g) $,
and $ R $ is the scalar curvature of $ (M, g) $.
\end{thm}


Our proof of Theorem \ref{evolution-bymass} makes use of a recent formula of Wang and Yau (Proposition 6.1 in \cite{WangYau08}):

\begin{prop} \label{propwangyau}
Let $ \Sigma $ be an orientable closed embedded hypersurface in $ \R^{n+1} $. Let $ \{ \Sigma_t \}_{ | t | < \delta } $ be a smooth variation of
$ \Sigma $ in $ \R^{n+1}$. Then
\be \label{eqwangyau}
\frac{d}{dt} \lf( \int_{\Sigma_t} H_0 \, d \sigma_t \ri) |_{t = 0}  =  \frac12 \int_{\Sigma}
\lf( H_0 \tr_{\Sigma} h -  \la A_0, h \ra \ri)  d \sigma ,
\ee
where $ H_0 $ and $ A_0 $ are the mean curvature and the second fundamental form of $ \S $ with respect to the outward normal
in $ \R^{n+1} $,  $ h $ is the variation of the induced metric $ \sigma $ on $ \Sigma $,
$ \tr_{\Sigma} h = \la \sigma , h \ra $ denotes the trace of $ h $ with respect to $ \sigma $,
and $ d \sigma_t $, $ d \sigma $ denote the volume form on $ \Sigma_t $, $ \Sigma$.

\end{prop}

In order to apply Proposition \ref{propwangyau}, we will need to show that, on a closed convex
surface $ \Sigma $ in $ \R^3 $,  an abstract  metric variation on $ \Sigma $ indeed arises from a
surface variation $ \{ \Sigma_t \}$  of $ \Sigma $ in $ \R^3 $.
Precisely, we have:

\begin{prop} \label{good-embedding}
Given an integer $ k \geq 6 $ and a number $ 0 < \alpha < 1 $,
let $ \{ \sigma(t) \}_{|t|<1}$ be a path of $ C^{k, \alpha} $ metrics on $ \mathbb{S}^2$
such that $ \{ \sigma(t) \} $ is differentiable at $ t = 0 $ in the space of $C^{k, \alpha}$
metrics. Suppose  $ \sigma(0) $ has positive Gaussian curvature.
Then there exists a small number $ \delta > 0$  and
a path  of $ C^{k, \alpha} $ embeddings $\{ f (t) \}_{| t | < \delta} $ of $ \mathbb{S}^2 $ in $ \R^3 $ such that
$ f(t) $ is an isometric embedding of $(\mathbb{S}^2, \sigma(t))$ for  $ |t| < \delta $
and $ \{ f(t) \} $ is differentiable at $ t=0$ in the space of
$ C^{2, \alpha}$ embeddings.
\end{prop}

The proposition above follows from the arguments by Nirenberg in \cite{Nirenberg}. For completeness, we
include its proof here.

\begin{proof}

Given $\{ \sigma (t) \}_{ | t |< 1} \}$, a path of $ C^{k, \alpha} $ metrics on $\mathbb{S}^2$, let $h=\sigma'(0)$.
Then $ h$ is a $C^{k, \alpha}$ symmetric (0,2) tensor. Since $ \sigma (0 ) $ has positive Gaussian curvature,
by the result in \cite{Nirenberg}, there exists a $ C^{k, \alpha}$ isometric embedding  of $(\mathbb{S}^2, \sigma(0) )$ in $ \R^3 $,
which we denote by $ X $.
Given such an $ X $, let $ Y : \mathbb{S}^2 \rightarrow \R^3 $  be a $C^{2, \alpha}$ solution to
the linear equation
\be \label{eq-linearembedding}
2dX\cdot dY=h ,
\ee
where $ ``  \cdot " $ denotes the Euclidean dot product in $ \R^3 $ and \eqref{eq-linearembedding} is understood as
$$ d X ( e_1) \cdot d Y ( e_2 ) + d X ( e_2 ) \cdot d Y ( e_1 ) = h ( e_1, e_2) $$
for any tangent vectors $ e_1$, $e_2 $ to $ \mathbb{S}^2 $.
The existence of such a $ Y$ is provided by Theorem 2' in \cite{Nirenberg}.
Let $ d \bar{\sigma}^2 = h $ and let $ \phi$, $ p_1$, $ p_2 $ be given as
in  (6.5),  (6.6) in \cite{Nirenberg}, then $ \phi $ satisfies
(6.15) in \cite{Nirenberg}. Using the fact that $ X$ is in $C^{k, \alpha}$ and
$ d \bar{\sigma}^2$ is in $C^{k, \alpha}$, we check that the coefficients of
(6.15) in \cite{Nirenberg} (when written in a non-divergence form) is in $C^{k-3, \alpha}$.
Thus, it follows from (6.15) in \cite{Nirenberg} that $ \phi \in C^{k-1, \alpha}$, from which
we conclude $ Y \in C^{k-1, \alpha}$ by (6.11)-(6.13) in \cite{Nirenberg}.

Now consider the $C^{k-1, \alpha}$ path of embeddings $\{ G(t) \}_{| t |< t_0}$,
where
\be
G(t) = X + t Y
\ee
and $t_0 $ is chosen so that $ G(t) $ is an embedding.
Let $ g_e $ be the Euclidean metric on $ \R^3$. The pull back metric
$\tau(t) = G(t)^*(g_e)$ which is in $C^{k-2, \alpha} $ satisfies
\be
\tau (0) = \sigma (0) , \ \
\tau^\prime (0) = \sigma^\prime (0) ,
\ee
which implies
\be
||\tau(t)-\sigma(t)||_{C^{2, \alpha} }=O(t^2) .
\ee
Apply Lemma \ref{lma-ap-3} in the Appendix to $ \sigma^0 = \sigma(0) = \tau(0)$, for each $ t $ sufficiently small,
we can find a $ C^{2, \alpha} $ isometric embedding $X(t)$ of $(\mathbb{S}^2, \sigma(t))$ in $ \R^3 $  such that
\be \label{eq-ap-quadratic}
||G(t)-X(t)||_{C^{2,\alpha}} \le C || \tau(t) - \sigma(t) ||_{C^{2, \alpha} } = O ( t^2 ).
\ee
(By Lemma 1' in \cite{Nirenberg}, $ X(t) $ indeed lies in $ C^{k, \alpha}$.)
It follows from \eqref{eq-ap-quadratic} that $ \{ X(t) \}$, when viewed as a path
in the space of $ C^{2, \alpha}$ embeddings, is differentiable at $ t=0 $. Proposition \ref{good-embedding} is therefore proved.
\end{proof}

Proposition \ref{propwangyau} and Proposition \ref{good-embedding} together imply:

\begin{prop} \label{gderivative}
Given an integer $ k \geq 6 $ and a number $ 0 < \alpha < 1 $,
suppose $ \{ \sigma(t) \}_{|t|<1}$ is a differentiable path in the space of
 $ C^{k, \alpha}$ metrics on $ \mathbb{S}^2$. Suppose  $ \sigma(t) $ has
 positive Gaussian curvature for each $ t $. Let $ H_0 $ be the mean curvature
of the isometric embedding of $(\mathbb{S}^2, \sigma(t))$ in $ \R^3 $ with respect to the outward normal. Let
$ d \sigma_t $ be the volume form of $ \sigma (t) $.
Then
$$
\int_{\mathbb{S}^2} H_0 \ d \sigma_t
$$
is a differentiable function of $ t $, and
\be \label{eqwangyau-2}
\frac{d}{dt} \lf( \int_{\mathbb{S}^2} H_0 \ d \sigma_t \ri)   =  \frac12 \int_{\mathbb{S}^2}
  \la H_0 \sigma(t) - A_0, h \ra \ra \ d \sigma_t ,
\ee
where  $ A_0 $ is the second fundamental form of the isometric embedding
of $(\mathbb{S}^2, \sigma(t))$ in $ \R^3$  with respect to the outward  normal,  $ h = \sigma^\prime (t) $,
$`` \la \cdot, \cdot \ra " $ denotes the metric product with respect to $\sigma (t) $ on the
space of symmetric $(0,2)$ tensors.
\end{prop}

\begin{proof}
Take any $ t_0 \in (-1, 1)$.
By Proposition \ref{good-embedding}, there exists a small positive number $ \delta$
(depending on $ t_0$) and a path of $ C^{k, \alpha}$
embeddings $ \{ f(t) \}_{ | t - t_0 | < \delta } $ of $ \mathbb{S}^2 $ in $ \R^3 $, such that
$ f(t) $ is an isometric embedding of $(\mathbb{S}^2, \sigma(t) )$
and $ \{ f(t ) \}$ is differentiable at $ t = t_0 $ in the space of
$C^{2, \alpha}$ embeddings.

Let $ \Sigma_t = f(t) (\mathbb{S}^2) $, let $ H_0 (t) $ be the mean curvature of $ \Sigma_t $ with respect to the outward normal
in $ \R^3 $,  by definition we have
\be \label{eq-realization}
 \int_{\mathbb{S}^2} H_0 \ d \sigma_t = \int_{\Sigma_t} H_0 (t) \ d \sigma_t, \ \ \forall \  | t - t_0 | < \delta .
\ee
Apply the fact that  $ \{ f(t ) \}$ is differentiable at $ t = t_0 $ in the space of $C^{2, \alpha}$ embeddings
and note that $ H_0$ only involves derivatives of $ f(t)$  up to the second order, we
conclude that $ \int_{\Sigma_t} H_0 (t) \ d \sigma_t $ is differentiable at $ t_0 $.
By \eqref{eq-realization},  $ \int_{\mathbb{S}^2} H_0 \ d \sigma_t$ is
differentiable at $ t_0$ as well.
This shows $  \int_{\mathbb{S}^2} H_0 \ d \sigma_t $ is a differentiable function of $ t$.
Equation  \eqref{eqwangyau-2} then follows directly from  \eqref{eq-realization}  and
Proposition \ref{propwangyau}.
\end{proof}

We are now ready to prove Theorem \ref{evolution-bymass} using Proposition \ref{gderivative}.

\vspace{.2cm}

\noindent {\em Proof of Theorem \ref{evolution-bymass}.}  By Proposition \ref{gderivative}, the  function
$ \mathfrak{m}_{_{\rm BY}} (\Sigma_t ) $ is a differentiable function of $ t $. We have
\be \label{eq-d1}
\begin{split}
\frac{d}{dt} \mathfrak{m}_{_{\rm BY}} ( \S_t) = & \ \frac{1}{ 8 \pi} \frac{ d}{dt}
\lf( \int_{\S_t} H_0 \ d \sigma_t \ri )
-\frac{1}{8 \pi} \frac{d}{dt} \lf( \int_{\S_t} H \ d \sigma_t \ri) .
\end{split}
\ee
Let $ \sigma = \sigma (t) $ be the induced metric on $ \S_t $.
By \eqref{eqwangyau-2} in Proposition \ref{gderivative}, we have
\be
\begin{split}
\frac{d}{dt} \lf( \int_{\S_t} H_0 \ d \sigma_t \ri) = \ &
 \frac12 \int_{\S_t}
  \la H_0 \sigma(t) - A_0, \frac{\p \sigma}{\p t}   \ra \ d \sigma_t .
\end{split}
\ee
Now, applying the fact that $ \{ \S_t \}$ evolves in $ (M, g)$ according to
\be
\frac{ \p F}{\p t} = \eta \nu ,
\ee
we have
\be \label{eq-c}
\frac{ \p \sigma}{\p t} = 2 \eta A ,
\ee
and
\be \label{eq-d}
\frac{ \p H}{ \p t} = -  \Delta \eta - ( | A |^2 + \Ric(\nu, \nu) ) \eta ,
\ee
where $ \Ric(\nu, \nu)$ is the Ricci curvature of $ (M, g) $ along $ \nu $.
Thus,
\be \label{eq-d2}
\begin{split}
\frac{d}{dt} \lf( \int_{\S_t} H_0 \ d \sigma_t \ri) = \ &
 \frac12 \int_{\S_t}
  \la H_0 \sigma(t) - A_0, 2 \eta A \ra \ d \sigma_t \\
  = \ &
  \int_{\S_t}
    H_0 H \eta  -   \la A_0,  A \ra \eta  \ d \sigma_t \\
\end{split}
\ee
and
\be \label{eq-d3}
\begin{split}
\frac{d}{dt} \lf( \int_{\S_t} H \ d \sigma_t \ri) = & \
\int_{\S_t} \frac{ \p H}{\p t} + H^2 \eta \ d \sigma_t \\
= &  \
 \int_{\S_t}  - ( | A |^2 + \Ric(\nu, \nu) ) \eta
+  H^2 \eta \ d \sigma_t .
\end{split}
\ee
Hence, it follows from \eqref{eq-d1}, \eqref{eq-d2} and \eqref{eq-d3} that
\be \label{eq-dbym2}
\begin{split}
 \frac{d}{dt} \mathfrak{m}_{_{\rm BY}} ( \S_t ) = \frac{1}{8 \pi} & \
\int_{\S_t}
   \lf[ H_0 H   -   \la A_0,  A \ra
    + ( | A |^2 + \Ric(\nu, \nu) ) -   H^2  \ri] \eta \ d \sigma_t .
\end{split}
\ee
Apply the Gauss equation to $ \Sigma_t $ in $ (M,g) $ and to the isometric embedding of $ \Sigma_t $ in $ \R^3 $ respectively, we have
\be \label{eq-gaussm}
2 K = R - 2 \Ric(\nu, \nu) + H^2 - | A|^2 ,
\ee
and
\be \label{eq-gausse}
2 K = (H_0)^2 - | A_0 |^2 ,
\ee
where $ K $ is the Gaussian curvature of $ \S_t $.
Hence, \eqref{eq-dbym2}-\eqref{eq-gausse} imply that
\be \label{eq-h}
\begin{split}
\frac{d}{dt} \mathfrak{m}_{_{\rm BY}} ( \S_t)
 = & \  \frac{1}{ 16 \pi} \int_{\S_t}
\lf[ | A_0 - A |^2 - ( H_0 - H )^2 + R \ri] \eta \ d \sigma_t .
\end{split}
\ee
Therefore, \eqref{formula-1} is proved.   \stop

\vspace{.3cm}

Next, we want to discuss some applications of Theorem \ref{evolution-bymass}.
The first two applications below put the monotonicity property of the Brown-York mass
in the construction in \cite{ShiTam02} into a more general context.

\begin{cor} \label{form2}
Let $ (M, g)$, $ I $, $ F $,  $ \{ \Sigma_t \}$, $ \eta$,  $ A$, $ A_0 $, $ H $ and $ H_0$ be given as in Theorem \ref{evolution-bymass}
with $ \eta > 0 $. Suppose at each point $x\in \Sigma_t$, $t\in I$, $A_0-A$ is either positive semi-definite or negative semi-definite,
and $R\le 0$, then
$\mathfrak{m}_{_{\rm BY}}(\Sigma_t)$ is nonincreasing in $t$.
If in addition, $ A = \alpha A_0 $ for some number $ \alpha $ depending on $ x \in \Sigma_t$,
then $\mathfrak{m}_{_{\rm BY}}(\Sigma_t)$ is constant in $I$ if and only if $(\mathbb{S}^2\times I, F^*(g))$ is a domain in $\R^3$.
\end{cor}
\begin{proof} Let $\lambda_1, \lambda_2$ be the eigenvalues of $A_0-A$. Suppose $A_0-A$ is either positive semi-definite or negative semi-definite, then $\lambda_1\lambda_2\ge 0$ and hence $|A_0-A|^2-|H_0-H|^2=-2\lambda_1\lambda_2\le 0$. Since $R\le 0$, by Theorem \ref{evolution-bymass}, we have:
$$
\frac{d}{dt}\mathfrak{m}_{_{\rm BY}}(\Sigma_t)\le0
$$
because $\eta>0$. This proves the first assertion.

Suppose $(\mathbb{S}^2\times I, F^*(g) )$ is a domain in $\R^3$, then by definition we have
$ \mathfrak{m}_{_{\rm BY}} (\Sigma_t) = 0 $, $ \forall t $. Hence,
$$
\frac{d}{dt}\mathfrak{m}_{_{\rm BY}}(\Sigma_t)=0.
$$
Conversely, suppose $ A = \alpha A_0 $ and
$$
\frac{d}{dt}\mathfrak{m}_{_{\rm BY}}(\Sigma_t)=0.
$$
Then $R=0$ and $A=A_0$. In particular, $H=H_0$. For any
$(t_1,t_2)\subset I$, let $\Omega= \mathbb{S}^2\times (t_1,t_2)$
with the pull back metric $F^*(g)$. Let $D$ be the interior of
$\Sigma_{t_{_1}} = F(\mathbb{S}^2 \times \{t_1\})$ when it is
isometrically embedded in $\R^3$ and $E$ be the exterior of
$\Sigma_{t_{_2}} = F(\mathbb{S}^2\times \{t_2\} )$ when it is
isometrically embedded in $\R^3$. By gluing $\Omega$ with $D$ along
$  \mathbb{S}^2 \times \{t_1\}$, which is identified with $
\Sigma_{t_{_1}} $ through $F$, and gluing $\Omega$ with $E$ along
$\mathbb{S}^2\times \{t_2\}$, which is identified with $
\Sigma_{t_{_2}} $ through $F$, we have an asymptotically flat and
scalar flat manifold with corners and with zero mass, and  it must
be flat by \cite{Miao02} \cite{ShiTam02}. Hence, $\Omega$ is flat.
Since it is simply connected, $\Omega$ can be isometrically embedded
in $\R^3$.

\end{proof}

\begin{cor} \label{form2-1}
Let $ (M, g)$, $ I $, $ F $,  $ \{ \Sigma_t \}$, $ \eta$,  $ A$, $ A_0 $, $ H $ and $ H_0$ be given as in Theorem \ref{evolution-bymass}.
Let $ g_e $ be the Euclidean metric on $ \R^3 $.
Suppose there exists another smooth map
$$
F^0 : \mathbb{S}^2 \times I \longrightarrow \R^3
$$
such that
\begin{enumerate}
\item[(i)] $ \Sigma^0_t = F^0(\mathbb{S}^2, t)$ is an embedded closed convex surface in $ \R^3 $ and
$$ (F^0_t)^* ( g_e ) = F_t^* ( g ) ,$$
where $ F^0_t ( \cdot ) = F^0 ( \cdot, t) $ and $ F_t ( \cdot ) = F ( \cdot , t)$.

\item[(ii)] The velocity vector $ \frac{\p F^0 }{\p t} $ is always perpendicular to $ \Sigma^0_t $, i.e
$$
\frac{ \p F^0}{\p t} = \eta^0 \nu^0 ,
$$
where $ \nu^0 $ is the outward unit normal to $ \Sigma^0_t $ in $ \R^3 $ and $ \eta^0 $
denotes the speed of $ \S^0_t $ with respect to $ \nu^0 $.
\end{enumerate}
Suppose $ \eta^0 > 0 $, $ \eta > 0 $ and $ (M, g)$ has zero scalar curvature, then
the Brown-York mass $ \mathfrak{m}_{_{\rm BY}} ( \S_t )$ is monotonically non-increasing, and
$ \mathfrak{m}_{_{\rm BY}} ( \S_t ) $ is a constant if and only if $  (\mathbb{S}^2 \times I, F^*(g) )$ is a domain in $\R^3$.
\end{cor}

\begin{proof}
Since  $ \eta^0 > 0 $ and $ \eta > 0 $, we can write $(F^0)^*(g_e)$  and $F^*(g)$ as
\be \label{gandgz}
F^* ( g_e ) = (\eta^0)^2 dt^2 + g_t
\ \
\mathrm{and}
\ \
F^* ( g )  = \eta^2 d t^2 + g_t,
\ee
where $ g_t  $ denotes the same induced metric on both $ \Sigma^0_t $ and $ \Sigma_t $.
Now it follows from \eqref{gandgz} that
\be \label{aandaz}
A = \frac{ \eta^0}{ \eta }  A_0 .
\ee
Since $A_0$ is positive definite, the results follow from Corollary \ref{form2}.
\end{proof}

\begin{rem}
We note that
\begin{enumerate}
  \item[(i)]  Quasi-spherical metrics constructed in \cite{ShiTam02} satisfy all the assumptions of Corollary \ref{form2-1}.
  \item[(ii)]   In case $\eta^0=1$, one recovers the monotonicity formula in \cite{ShiTam02}.
\end{enumerate}
\end{rem}

By applying the co-area formula directly to \eqref{formula-1}, we
also obtain

\begin{cor}
Let $ (M, g)$, $ F $, $ \{ \Sigma_t \}$, $ \eta$,  $ A$, $ A_0 $, $ H $ and $ H_0$ be given as in Theorem
\ref{evolution-bymass}. Suppose $ \eta > 0 $. For any $ t_1 < t_2 $, let
$ \Omega_{[t_1, t_2] } $ be the region bounded by $ \S_{t_1} $
and $ \S_{t_2}$.
Then
\be \label{formula-2}
\mathfrak{m}_{_{\rm BY}} ( \S_{t_2}  ) - \mathfrak{m}_{_{\rm BY}} ( \S_{t_1} ) =
\frac{1}{16 \pi}
\left(
\int_{ \Omega_{ [ t_1, t_2 ] } }  R \ d V
+ \int_{ \Omega_{ [ t_1, t_2 ] } } \Phi \ d V \right) ,
\ee
where $ R $ is the scalar curvature of $ (M, g) $, $ d V $ is the volume form of $ g $ on $ M $, and
 $ \Phi $ is the function on $ \Omega_{ [ t_1, t_2 ] } $,
depending on  $ \{ \Sigma_t \}$,  defined by
\be
\Phi ( x ) = |  A_0 - A |^2 - ( H_0 - H )^2, \
 x \in \S_t .
\ee
\end{cor}

The function $ \Phi (x )$ defined above clearly depends on the foliation $ \{ \Sigma_t \} $
connecting $ \Sigma_{t_1} $ to $ \Sigma_{t_2} $. However, it is interesting to
note that the integral
$ \int_{ \Omega_{[ t_1, t_2] } } \Phi \ d V $
turns out to be $\{ \Sigma_t \}$ independent by \eqref{formula-2}.

We can apply formula \eqref{formula-2}  to small geodesic balls in a general $3$-manifold and to
asymptotically flat regions in an asymptotically flat $3$-manifold.

\begin{cor}
Let $ (M, g) $ be a $ 3$-dimensional Riemannian manifold. Let $ p \in M $ and
$ B_\delta (p) $ be a geodesic ball centered at $ p $ with geodesic radius $ \delta $.
Suppose $ \delta $ is small enough such that
\begin{enumerate}
\item $ \delta < i_p (M ) $, where $ i_p (M) $ is the injectivity radius of $(M, g) $ at $ p $.
\item For any $ 0 < r \leq \delta $, the geodesic sphere $ S_r(p)$, centered at $ p $ with
geodesic radius $ r $, has positive Gaussian curvature.
\end{enumerate}
Then the Brow-York mass of $ S_\delta (p)$ can be written as
\be  \label{formula-3}
\mathfrak{m}_{_{\rm BY}} ( S_\delta (p)  )  =
\frac{1}{16 \pi}
\left(
\int_{ B_\delta(p) }  R \ d V
+ \int_{ B_\delta (p) \setminus \{ p \}  } \Phi \ d V \right) ,
\ee
where  $ R $ is the scalar curvature of $ M $, $ d V $ is the volume form on $ M $, and
 $ \Phi $ is the function on $ B_\delta (p) \setminus \{ p \} $, defined by
\be
\Phi ( x ) = |  A_0 - A |^2 - ( H_0 - H )^2, \
 x \in S_r .
\ee
Here $ A$, $ H$ are the second fundamental form, the mean curvature of $ S_r $ in $ M $ with respect to
the outward normal; and $ A_0 $, $ H_0 $ are the second fundamental form, the mean curvature of
the isometric embedding of $ S_r $ in $ \R^3 $ with respect to the outward normal.
\end{cor}

\begin{proof}
Let $ (r, \omega)$ be the geodesic polar coordinate of $ x \in B_\delta(p) \setminus \{ p \} $,
where $ r $ denotes the  distance from $ x $ to $ p $.
Since $ \frac{ \p }{\p r} \perp S_r$, we can choose the foliation $ \{ \S_t \} $ in Corollary \ref{form2}
to be $ \{ S_r \}$ with $ t = r $. By \eqref{formula-2}, we have
\be \label{eq-j}
\mathfrak{m}_{_{\rm BY}} ( S_\delta (p) )   - \mathfrak{m}_{_{\rm BY}} ( S_r (p) ) =
\frac{1}{16 \pi}
\int_{ B_\delta(p) \setminus B_r (p) }  ( R + \Phi ) \ d V .
\ee
By \cite{FanShiTam07}, we have
\be \label{eq-k}
\lim_{r \rightarrow 0+} \mathfrak{m}_{_{\rm BY}} ( S_r (p) ) = 0 .
\ee
Hence, \eqref{formula-3} follows from \eqref{eq-j} and \eqref{eq-k}.
\end{proof}


Next, we express the ADM mass \cite{ADM61}
as the sum of the Brown-York mass of a coordinate sphere and an integral involving the
scalar curvature and the function $ \Phi (x) $.

\begin{cor} \label{ADMmassformula}
Let $ (M, g)$ be an asymptotically flat $ 3$-manifold
 with a given end. Let $ \{ x^i \ | \ i = 1,2,3 \}$ be a coordinate system
 at $ \infty $ defining the
 asymptotic structure of $ (M, g )$.
 Let  $ S_r = \{ x \in M \ | \ | x | = r \}  $
 be the coordinate  sphere, where $ | x | $
 denotes the coordinate  length.
 Suppose $ r_0 \gg 1 $ is a constant such that $ S_r $ has positive
 Gaussian curvature for each $ r \geq r_0$.
 Then
  \be \label{formula-adm}
 \mathfrak{m}_{_{\rm ADM}} = \mathfrak{m}_{_{\rm BY}} ( S_{r_0} ) + \frac{1}{16 \pi}
\int_{ M \setminus D_{r_0} }  R \ d V + \frac{1}{16 \pi} \int_{ M
\setminus D_{r_0} }  \Phi \ d V ,
 \ee
 where $ \mathfrak{m}_{_{\rm ADM}} $ is the ADM mass of $ (M, g)$,
 $ R $ is the scalar curvature of $ (M, g)$,
 $ D_{r_0} $ is the bounded open set in $ M $ enclosed by $ S_{r_0}$,
  and
 $ \Phi $ is the function on $ M \setminus D_{r_0} $ defined by
\be \Phi ( x ) = |  A_0 - A |^2 - ( H_0 - H )^2, \
 x \in S_r .
\ee
Here $ A$, $ H$ are the second fundamental form, the mean curvature of $ S_r $ in $ M $;
and $ A_0 $, $ H_0 $ are the second fundamental form, the mean curvature of
$ S_r $ when isometrically embedded in $ \R^3$.
\end{cor}

\begin{proof}
$\{ S_r \}_{r \geq r_0} $ consists of level sets of the function $ r
$ on $ M \setminus D_{r_0}$, hence can be {\em reparameterized} to
evolve in a way  that its velocity vector is perpendicular to the
surface at each time. To be precise, we can
 define the vector field $ X = \frac{ \nabla r }{ | \nabla r|^2 } $ on $ M \setminus D_{r_0}$
 and let $ \gamma_p (t) $ be the integral curve of $ X $  starting at $ p \in S_{r_0}$.
 For any $ t \geq 0 $, let
 $ \Sigma_t  = \{  \gamma_p (t) \ | \ p \in S_{r_0} \}$, then $ \Sigma_t = S_{r_0 +t} $.
For any $ T > 0 $, apply \eqref{formula-2} to $\{ \Sigma_t \}_{ 0
\leq t \leq T }$, we have \be  \label{eq-l} \mathfrak{m}_{_{\rm BY}}
( S_{r_0 +T}  ) - \mathfrak{m}_{_{\rm BY}} ( S_{ r_0} ) =
\frac{1}{16 \pi} \left( \int_{ \Omega_{[0,T]} }  R \ d V + \int_{
\Omega_{[0,T] } } \Phi \ d V \right) , \ee where $ \Omega_{[0,T]} $
is the region in $ M $  bounded by $ \S_0 = S_{r_0}$ and $ \S_T =
S_{r_0 +T}$. Letting $ T \rightarrow + \infty$, by
\cite{FanShiTam07} we have \be \label{eq-m} \lim_{T \rightarrow +
\infty} \mathfrak{m}_{_{\rm BY}} ( S_{T} ) = \mathfrak{m}_{_{\rm
ADM}} . \ee Hence, \eqref{formula-adm} follows  from \eqref{eq-l}
and \eqref{eq-m}.
\end{proof}

\section{Liu-Yau mass of spacelike two-surfaces in $\mathbb{R}^{3,1}$}

Let $\Sigma$ be a closed, connected, $2$-dimensional spacelike surface in a spacetime
$N$.  Suppose $ \Sigma $ has positive Gaussian curvature and has spacelike mean curvature
vector $ \vec H $ in $ N$.
Let $ H_0 $ be the mean curvature of $ \Sigma $ with respect to the outward unit normal
when it is isometrically embedded in $ \R^3 $. The Liu-Yau mass
of $ \Sigma $ is then defined as (see \cite{LY1, LY2}):
$$ \mathfrak{m}_{_{\rm LY}}(\Sigma)=\frac{1}{8\pi }\int_\Sigma (H_0 -|\vec H|) \ d\sigma, $$
where  $|\vec H|$ is Lorentzian norm of $\vec H$ in $ N $ and $ d \sigma $ is the volume
form of the induced metric on $ \Sigma $.

In \cite{LY2}, the following positivity result was proved: {\it Let $\Omega$ be a compact,
spacelike hypersurface in a spacetime $N$ satisfying the dominant energy conditions.
Suppose the boundary $\p \Omega$ has finitely many components $\Sigma_i$, $1 \le i\le l$,
each of which has positive Gaussian curvature and has spacelike mean curvature vector in $ N$.
Then $\mathfrak{m}_{_{\rm LY}}(\Sigma_i)\ge0$ for all $i$; moreover if $\mathfrak{m}_{_{\rm LY}}(\Sigma_i)=0$
for some $i$, then $\p\Omega$ is connected and $N$ is a flat spacetime along $ \Omega $}.

We remark that in their proof of the above result, it is assumed implicitly that the mean curvature
of $\p\Omega$ in $\Omega$ with respect to the outward unit normal is positive. See the statement \cite[Theorem 1.1]{WangYau07}.
The condition is necessary as  can be seen by the following example
in the time symmetric case:

Let $ g_e $ be the Euclidean metric on $ \R^3 $ and let $ m > 0 $ be a constant.
Consider the Schwarzschild metric (with negative mass)
$$
g = \lf(1 -\frac m{2 | x | } \ri)^4 g_e,
$$
defined on $\{ 0< | x | < \frac m 2 \} $. Given any $ 0 < r_1 < r_2 < \frac{m}{2} $,
consider the domain
$$ \Omega = \lf\{r_1 < | x |  <r_2 \ri\} .$$
For any constant $ r $, the mean curvature $H$ of the sphere
$ S_r = \{ | x | = r \}$ with respect to the unit normal in the direction of $\p/\p r$ is
$$
H=\frac1{(1 -\frac m{2r})^2}\lf(\frac 2r+\frac{4}{1 -\frac m{2r}}\frac{m}{2r^2}\ri).
$$
The mean curvature of $ S_r $ when it is embedded in $\R^3$ is
$$
H_0=\frac1{(1 -\frac m{2r})^2}\frac 2r.
$$
Suppose $r<\frac m2$, then $H<0$ and
$$
|H|-H_0=-\frac4{r(1 -\frac m{2r})^3}>0.
$$
Hence, $\mathfrak{m}_{_{\rm LY}}(S_{r_{_1}} )<0$ and $\mathfrak{m}_{_{\rm LY}}(S_{r_{_2}})<0$ where
$ \p \Omega =  S_{{ r_{_1}}} \cup S_{{r_{_2}}} $.

In \cite{MST}, \'{O}  Murchadha,   Szabados and Tod gave some examples of
a spacelike $2$-surface, lying on the light cone of the Minkowski space
$\mathbb{R}^{3,1}$, whose Liu-Yau mass is strictly positive.
Motivated by their result, we want to understand the Liu-Yau mass of more general
spacelike $2$-surfaces in $ \R^{3,1}$. In the sequel, we always regard
$\mathbb{R}^3$ as the $t=0$ slice in $\mathbb{R}^{3,1}$.  We have the following:

\begin{thm}\label{LiuYau-t1} Let $\Sigma$ be a closed, connected, smooth, spacelike
$2$-surface in $\R^{3,1}$. Suppose $ \Sigma $ spans a compact spacelike hypersurface
in $ \R^{3,1}$.
If $ \Sigma $ has positive Gaussian curvature and has spacelike mean curvature vector, then
$\mathfrak{m}_{_{\rm LY}}(\Sigma)\geq 0$; moreover $ \mathfrak{m}_{_{\rm LY}} (\Sigma ) = 0 $
if and only if $\Sigma$ lies on a hyperplane in $\mathbb{R}^{3,1}$.
\end{thm}

In order to prove this theorem,  we need the following result which can be proved by the method of
Bartnik and Simon \cite{BartnikSimon82} and by an idea  from Bartnik \cite{Bartnik09}. In fact, it is just a special case of the results by Bartnik \cite{Bartnik88}.

\begin{lma}\label{LiuYau-l1}
Let $\Sigma$ be a closed, connected, smooth, spacelike $2$-surface in $\R^{3,1}$.
Suppose $ \Sigma $ spans a compact spacelike hypersurface  in $\R^{3,1}$. Then $\Sigma$ spans
a compact, smoothly immersed, maximal spacelike hypersurface in $\R^{3,1}$.
\end{lma}

\begin{proof}
Let $ M $ be a compact spacelike hypersurface in $ \R^{3,1} $ spanned by $ \Sigma $.
By extending $M$ a bit, we may assume that there exists a spacelike hypersurface $ \tilde M $ in $ \R^{3,1} $
such that  $ \ol M \subset \tilde M $. Since $ \tilde M $ is spacelike, $ \tilde M $ is {\it locally} a graph
over an open set in $ \R^{3}$. Hence, the projection map $ \pi : \tilde M \rightarrow \R^3 $,
given by $ \pi ( x, t ) = x $, is a local diffeomorphism. Now consider the map
\be
F : \tilde M \times \R^1 \longrightarrow \R^{3,1},
\ee
given by $ F (p, s) = ( x, s ) $ for any $ p = ( x, t) \in \tilde M $,
then $ F $ is a local diffeomorphism as well. Let $ N = \tilde M \times \R^1 $ equipped with
the pull back metric. Let $ v $ be the time function on $ \tilde M $ in $ \R^{3,1} $,
i.e. $ v ( x, t) = t $. Since $ v $ is a smooth function on $ \tilde M $, we
can consider its graph in $ N $.
Let $ \hat \Sigma $ and $ \hat G $ be the graph of $ v $ over $ \Sigma $ and $ \ol M $
in $ N $ respectively. Then $ \hat G $ is  a compact, spacelike hypersurface in $ N $ whose boundary is $ \hat \Sigma $.
Moreover, $ F|_{ \hat G} : \hat G \rightarrow \ol M \subset \R^{3,1} $ and
$ F |_{ \hat \Sigma } : \hat \Sigma \rightarrow \Sigma \subset \R^{3,1} $  are both isometries.

Now one can carry over the arguments in section 3 in \cite{BartnikSimon82} to prove that there is
a smooth solution (defined on $ M $) to the maximal surface equation in $ N $  such that,
if $ G $ is its graph in $ N$, then $ \p G = \hat \Sigma $.
For example, Lemma 3.3 in \cite{BartnikSimon82} can be rephrased as: For $\theta>0$, let
$$
\mathcal{D}=\{ \phi \in C^{0,1}(\ol M)|\ |D\phi|\le (1-\theta)\}
$$
and
\be
\begin{split}
\mathcal{F} = & \ \{u\in C^2(  M) \ | \ |Du|<1  \text{ with maximal graph} \\
& \ \hspace{2.6 cm} \text{ and $u=\phi$ on $ \Sigma $ for some $\phi \in \mathcal{D}$}\}.
\end{split}
\ee
Then there exists $r_0>$ and $\theta_1>0$ such that for all $u\in \mathcal{F}$ and for all
$p, q \in M$ with $d(p,\Sigma ), d(q, \Sigma )< \frac13r_0$(say) and $d(p,q)= r_0$ we have: $|u(p)-u(q)|\le (1-\theta_1)r_0$.

One then readily checks that $F ( G ) $ is a compact, smoothly immersed, maximal hypersurface in $ \R^{3,1} $ spanned by  $ \Sigma =  F (\hat \Sigma)$.
\end{proof}

To prove Theorem \ref{LiuYau-t1}, we also need a technical lemma concerning the boundary mean curvature of a compact spacelike hypersurface
in $ \R^{3,1}$, whose boundary has spacelike mean curvature vector.

\begin{lma} \label{bdrylemma}
Let $ M $ be a compact $3$-manifold with boundary $ \p M $. Let $ F : M \rightarrow \R^{3,1} $ be
a smooth, maximal spacelike immersion such that  $ F|_{\p M}  : \p M \rightarrow F ( \p M ) \subset \R^{3,1} $
has spacelike mean curvature vector.   Let $ g $ be the pull back metric on $ M $ and
$ k $ be the mean curvature of $ \p M $ in $ (M, g) $ with respect to the outward unit normal.
Then $ k $ must be nonnegative at some point on $ \p M $.
\end{lma}

\begin{proof}
Suppose $k<0$ everywhere on $\p M $. Since $ {F(M) }$ is a compact subset in $ \R^{3,1}$, without loss of generality,
we may assume that $F(M) \subset \{x_1\le 0\}$ and $ {F(M)}\cap \{x_1=0\}\neq\emptyset$.
Let $X_0=F(q)\in {F(M)}\cap \{x_1=0\}$ for some $ q \in M $. If $q$ is an interior point of $M$,
then there exists an open neighborhood $ V $ of $ q $ in the interior of $ M $ such that the tangent space
of $F( V )$ at $ X_0  $ is $\{x_1=0\}$. This is impossible, because $ F (V) $ needs to be spacelike.
Therefore, $q\in \p M$. Using the fact that $ F $ is a spacelike immersion again, we know there exists
 an open neighborhood $U $ of $ q $ in $ M $ such that $ F( \ol U )$ is a spacelike graph of some
 function $ f $ over $ \ol D $ for some open set $ D \subset \R^3 \cap \{ x_1 \leq 0 \} $.
Let $B=F(\ol U\cap \p M)  $ and let $ \hat{ B } $ be the part of $ \p D $ such that
$ B $ is the graph of $ f $ over $ \hat{ B }$.  We note that $ X_0 \in B $.
Without loss of generality, we may assume that $X_0$ is the origin.

To proceed, we let $T=\frac{\p }{\p t}$  and define the following notations:

  $n$: the future time like unit normal to $F(\ol U)$ in $ \R^{3,1}$;

  $\nu$: the unit outward normal to $ B $ in $F(\ol U)$;

  $\hat\nu$: the unit outward normal to $\hat{ B }$ in $\ol D$.

\noindent We parallel translate $\nu$, $ \hat\nu$  and all the tangent vectors of $B$, $ \hat{B} $ along the $T$ direction.
Also, we consider $f$ as a function on $D\times (-\infty,\infty)$ so that $ f $ is independent of $t$.

Now $\hat\nu$ is  normal to $ B $, so $\hat\nu=u\nu+vn$ for some numbers $ u, v $ satisfying $u^2-v^2=1$.
At $X_0 \in B$, we have $\hat\nu= \frac{\p}{\p x_1} $. Suppose $\alpha( s )=(x_1( s ), x_2(s ), x_3(s ), t(s))$
is a curve in $F(\ol U)$ such that $\alpha(0)=X_0$ and $\alpha'(0)=\nu$. Then, for $t<0$ small,
$\alpha(t)\in F(\ol U)$ and so $x_1(t)<0$. Since $x_1(0)=0$, we have $x_1'(0)\ge 0$, hence
$ u = \la \hat\nu,\nu\ra\ge0$.  Since $u^2=1+v^2$, we have $u> | v | $ at $ X_0 $.

Let $ \vec H $ be the mean curvature vector of $ B $ in $ \R^{3,1}$. Let $ p=p_{ij}  $ be the second
fundamental form of $ F( \ol U )$ in $ \R^{3,1} $ with respect to $ n $. Then
\be
 \vec H=-k\nu + ( \tr_{_B} p ) n,
\ee
where $ \tr_{_B} p $ denotes the trace of $ p $ restricted to $ B $.
Hence,
\be
   -\la \vec H, \hat\nu\ra= uk + v ( \tr_{_B} p) .
\ee
At $ X_0 $, we have shown $u>|v|$. On the other hand, we know $|k|>|\tr_{_B} p|$ (because $ \vec H $ is spacelike) and $k<0$ (by the assumption),
therefore we  have $-\la\vec H, \hat\nu\ra<0$ at $X_0$.
Recall that
\be
\la \vec H, \hat\nu\ra = \la \sum_{i=1}^2 \nabla_{e_i}e_i, \nu \ra ,
\ee
 where $\{e_1, e_2\}$ is an orthonormal frame in $ T_{_{X_0}} B $ and $ \nabla $ is the covariant derivative in $\R^{3,1}$. Hence there exists a unit vector $e\in T_{_{X_0}}B $ such that
\be \label{eq-speciale}
-\la \nabla_e e, \hat\nu\ra<0.
\ee
Suppose $e$ is the tangent of a curve $\gamma(s) \subset B $ at $s=0$.
Let $\hat\gamma(s) \subset \hat B $ be the projection of $ \gamma (s) $ in $\R^3$.
Then
\be
\gamma^\prime (s) = \hat \gamma^\prime (s) + \frac{d}{ds} f ( \hat \gamma (s) ) T .
\ee
Hence,
\begin{equation} \label{eq-gammaandgammap}
   \begin{split}
        -\la \nabla_{\gamma^\prime (s) } \gamma^\prime (s) , \hat\nu\ra
         & =-\la \nabla_{\gamma^\prime (s) }  \hat \gamma^\prime (s) , \hat\nu\ra \\
         &=-\la \nabla_{\hat \gamma^\prime (s) } \hat \gamma^\prime (s) , \hat\nu\ra ,
     \end{split}
\end{equation}
where we have used the facts that $ T $ is parallel, $ T \perp \hat \nu $ and $ \hat \gamma^\prime (s) $
is parallel translated along $ T $. Thus, it follows from \eqref{eq-speciale}, \eqref{eq-gammaandgammap}
and the fact $ e = \gamma^\prime (0) $ that
\be
-\la \nabla_{\hat \gamma^\prime ( 0 ) } \hat \gamma^\prime ( 0) , \hat \nu \ra < 0 .
\ee
But this is impossible because $ \hat \gamma (s) \subset \hat B \subset \{x_1 \le 0\} \cap \R^3$ and $\hat \nu= \frac{ \p }{\p x_1} $
at $ \hat \gamma (0) $.  Therefore, we have proved that $ k $ can {\it not} be negative everywhere on $ \p M$.
\end{proof}

\begin{proof}[Proof of Theorem \ref{LiuYau-t1}]
By Lemma \ref{LiuYau-l1}, we know that $\Sigma$ indeed bounds a compact, smoothly immersed,  maximal spacelike  hypersurface in $\mathbb{R}^{3,1}$.
Precisely, this means that there exists a compact $ 3$-manifold $M$ with boundary $ \p M $ and a smooth, maximal spacelike immersion
$F: M \to \R^{3,1}$ such that $F:\p M\to \Sigma$ is a diffeomorphism.

Let $ g = g_{ij}  $ be the pull back metric on $ M $. Let $ p = p_{ij} $ be the
second fundamental form of the immersion $ F : M \rightarrow \R^{3,1} $.
Let $ R $ be the scalar curvature of $ (M, g)$.
Since $ F $ is a maximal immersion, it follows from the constraint equations (or simply the Gauss equation) that
\be
R = | p |^2 \geq 0,
\ee
where $ `` | \cdot | " $ is taken with respect to $ g $.
On the other hand, let $ k$ be the mean curvature of $ \p M $ in $ (M, g)$
with respect to the outward unit normal and let $ \vec H $ be the mean curvature vector of
$ \Sigma = F ( \p M ) $ in $ \R^{3,1} $,
it is known that
\be \label{eq-spacelikeH}
| \vec H |^2 = k^2 - ( \tr_{_\Sigma } p )^2 ,
\ee
where  $ \tr_{_\Sigma} p $  is the trace of $ p $ restricted to $ \Sigma $. Since $ \vec H $ is spacelike,
\eqref{eq-spacelikeH} implies that either $ k > 0 $ or $ k < 0 $ on $ \p M $ because
 $\p M$ is connected. By Lemma \ref{bdrylemma}, we have $ k  > 0 $ on $ \p M $.

Now let $ k_0 $ be the mean curvature of $ \Sigma $ with respect to the unit outward
normal when it is isometrically embedded in $ \R^3 $.
It follows from \eqref{eq-spacelikeH} that
\be \label{liuyauby}
\int_\Sigma(k_0-|\vec H|) \ d\sigma\ge \int_\Sigma (k_0-k) \ d\sigma .
\ee
On the other hand, by the result of \cite{ShiTam02}, we have
\be \label{shitammass}
\int_\Sigma (k_0-k) \ d\sigma\ge0
\ee
and equality  holds if and only if $(M, g) $ is a domain in $\R^3$. Thus, we
conclude from \eqref{liuyauby} and \eqref{shitammass} that $ \mathfrak{m}_{_{\rm LY}} ( \Sigma ) \geq 0 $. Moreover,  if $ \mathfrak{m}_{_{\rm LY}} ( \Sigma ) = 0 $,
then $(M, g)$ must be flat, hence  $ R = 0 $ and consequently $ p   =0$. Therefore, $ F(M) $ and hence $ \Sigma $
lie on a hyperplane in $\R^{3,1}$. Conversely, if $\Sigma$ lies on a hyperplane in $\R^{3,1}$, then obviously $ \mathfrak{m}_{_{\rm LY}} ( \Sigma ) = 0 $.
\end{proof}

In the sequel, we want to show that the examples given in \cite{MST}
satisfy the assumption in Theorem \ref{LiuYau-t1}. To do that, we need the following definition:

\begin{defn}\label{acausal}
Two points $p$ and $q$ in a Lorentzian manifold $N$ are said to be
{\it causally related} if $p$ and $q$ can be joined by a timelike or
null  path. A set $S$ in
 $N$ is called {\it acausal} if no two points in $S$
are causally related.
\end{defn}

We claim that {\it all surfaces in the examples in \cite{MST} are acausal}.
Suppose this claim is true, then by Theorem 3 in \cite {F}(P.4765),
we know that those surfaces span spacelike hypersurfaces in
$\mathbb{R}^{3,1}$, hence satisfying the assumption in Theorem
\ref{LiuYau-t1}.

To verify the claim, let $\Sigma$ be an example given in \cite{MST}, i.e.
in terms of the usual spherical coordinates $(t, r, \theta, \phi)$ in $ \R^{3,1}$,
$ \Sigma $ is determined by the equation
\be  t = r = F (\theta, \phi) ,
\ee
where $ F = F ( \theta , \phi)$ is a smooth positive function of $ (\theta, \phi) \in \mathbb{S}^2 $.
Suppose $ \Sigma $ is {\it not acausal}, then there exists two distinct points $ p $, $ q $ in
$ \Sigma $ and a path $ \gamma (\tau )$ in $ \Sigma $ such that  $\gamma(0)=p$,
$\gamma(1)=q$, and
\be \label{nonspacelike}
(\dot{x})^2 +(\dot{y})^2 +(\dot{z})^2 \leq (\dot{t})^2,  \ \ \forall \ \tau \in [0,1] ,
\ee
here we denote $\gamma(\tau)=(x(\tau), y(\tau), z(\tau), t(\tau))$
and $``  \ \dot{} \  " $ denotes the derivative with respect to $\tau$.
Since  $\dot {\gamma}\neq 0$, without loss of generality,
we may assume $\dot{t}>0$. Let
$ r=\sqrt{x^2 +y^2 + z^2},$ \eqref{nonspacelike} implies that
\be
 |\dot{r}|\leq \dot{t}.
\ee
Note that $r(0)=t(0)$ and $r(1)=t(1)$, we see that
\be
|\dot{r}|=\dot{t},
\ee
for all $ \tau \in [0, 1]$.
By the equality case in the Cauchy-Schwartz inequality, we must have
\be x=k(\tau)\dot{x}, \  y=k(\tau)\dot{y}, \ z=k(\tau)\dot{z}
\ee
for some function $ k = k (\tau )$ and for all $ \tau \in [0,1]$.
Clearly, this implies that $p$ and $q$  lie on a line which passes
through the origin, or equivalently, $p=aq$ for some  positive number
$a$.  On the other hand, using the Cartesian coordinates,
we may write $p$  as
$$ (F(\theta_1,
\phi_1),F(\theta_1, \phi_1)\sin\theta_1\cos\phi_1, F(\theta_1,
\phi_1)\sin\theta_1\sin\phi_1, F(\theta_1, \phi_1)\cos\theta_1)$$
and write $ q $ as
$$ (F(\theta_2, \phi_2),F(\theta_2, \phi_2)\sin\theta_2\cos\phi_2,
F(\theta_2, \phi_2)\sin\theta_2\sin\phi_2, F(\theta_2,
\phi_2)\cos\theta_2) $$
for some $ (\theta_i, \phi_i) \in \mathbb{S}^2 $, $i =1, 2$.
The fact $p= aq$, for some $a>0$, then implies $p=q$, which is
contradiction. Therefore, $ \Sigma $ is {\it acausal}.

\

\

\section{Appendix}

In this appendix, we give some Lemmas which are needed to
complete the proof of Proposition \ref{good-embedding}.
We will follow closely Nirenberg's argument in \cite{Nirenberg}.
First, we introduce some notations:
given an integer $ k \geq 2 $ and a positive number $ \alpha < 1 $, let
$$
\begin{array}{lll}
 \E^{k, \alpha} & = & \mathrm{the \ space \ of } \ C^{k,\alpha} \ \mathrm{embeddings \ of} \ {S}^2
\ \mathrm{into}  \ \R^3 \\
 \X^{k, \alpha} & =  & \mathrm{the \ space \ of \ } C^{k,\alpha} \ \R^3
 \mathrm{-valued \ vector \ functions \ on \ } {S}^2 \\
 \Sp^{k, \alpha} & = & \mathrm{the \ space \ of \ } C^{k,\alpha} \  \mathrm{symmetric \ (0,2) \ tensors \ on}
 \ {S}^2 \\
 \M^{k, \alpha} & = & \mathrm{the \ space \ of \ } C^{k,\alpha} \  \mathrm{Riemannian \  \ metrics \ on}
 \ {S}^2 .
\end{array}
$$

On page 353 in \cite{Nirenberg}, Nirenberg proved

\begin{lma} \label{lma-ap-1}
Let  $ \sigma \in \M^{4, \alpha}$ be a metric with positive
Gaussian curvature. Let $ X \in \E^{4, \alpha}$ be an
isometric embedding of $(\mathbb{S}^2, \sigma)$ in $ \R^3$.
There exists two positive numbers $ \epsilon $
and $ C$, depending only on $ \sigma$,
such that if $ \tau \in \M^{2, \alpha}$ satisfying
$$ || \sigma - \tau ||_{C^{2, \alpha}} < \epsilon ,$$
then there is an isometric embedding $ Y \in \E^{2, \alpha}$ of $(\mathbb{S}^2, \tau)$
in $ \R^3 $ such that
$$ || X - Y ||_{C^{2, \alpha} } \leq C || \sigma - \tau ||_{C^{2, \alpha}} .$$
\end{lma}

In what follows, we want to show that the constants $ \epsilon $ and $ C$ in
the above Lemma can be chosen to be independent on $ \sigma $, provided
$ \sigma $ is sufficiently close to some $ \sigma^0 \in \M^{5, \alpha}$
(see Lemma \ref{lma-ap-3}). First, we prove the following:

 \begin{lma}\label{lma-ap-2}
Let  $ \sigma^0 \in \M^{5, \alpha}$ be a metric with positive
Gaussian curvature. There exists positive numbers $ \delta$ and $ \hat{K}$,
depending only on $ \sigma^0$,
such that if  $ \sigma \in \M^{4, \alpha}$ satisfying
$$ || \sigma^0 - \sigma ||_{C^{2, \alpha}} < \delta , $$
then for any $\gamma \in \Sp^{2, \alpha}$ and any  $ Z \in \X^{2, \alpha}$,
there exists a solution $ Y \in \X^{2, \alpha} $ to the linear equation
\be \label{eq-ap-linear}
2 d X^\sigma \cdot d Y = \gamma - (d Z )^2  .
\ee
Here $ X^\sigma \in \E^{4, \alpha}$ is any given isometric
embedding of $ (\mathbb{S}^2, \sigma)$.
Moreover, for  every $ Z $ (with $ \gamma $ fixed), a particular solution
$ Y $ denoted by $ \Phi (Z) $ may be chosen so that
 \be \label{eq-s57}
 ||\Phi(Z)||_{C^{2, \alpha}}\le \hat{K} \lf(||\gamma||_{C^{2, \alpha} }+||Z||_{C^{2, \alpha}}^2\ri),
 \ee
 and  for  any $Z, Z^1 \in \E^{2, \alpha}$,
 \be \label{eq-s58}
 ||\Phi(Z)-\Phi(Z_1)||_{C^{2,\alpha}}\le \hat{K} ||Z+Z_1||_{C^{2, \alpha}} \cdot ||Z-Z_1||_{C^{2, \alpha} } .
 \ee
\end{lma}

\begin{proof}  We proceed exactly as in \cite{Nirenberg}.
For any $ \sigma \in \M^{4, \alpha}$, let $ X^\sigma$ be a given
isometric embedding of $(\mathbb{S}^2, \sigma)$ in $ \R^3$.
 Let  $ X_3 $ be the unit inner normal to the surface $ X^\sigma (\mathbb{S}^2 )$.
 Let $\{u, v\}$ be a local coordinate chart on $\mathbb{S}^2$.
 Let $\phi$, $p_1$, $p_2$ be defined as in  (6.5),  (6.6) in \cite{Nirenberg}.
 Then $ \phi$, $p_1$, $p_2 $ satisfy the system of equations (6.11)-(6.13) in \cite{Nirenberg}
 with $ c_1$, $c_2$, $ \Delta$ defined  on
 page 356-357 in \cite{Nirenberg}.
By section 6.3 in \cite{Nirenberg}, the derivatives of $Y$
are completely determined by $ \phi $, which satisfies (6.15) in \cite{Nirenberg}.
Let $ \phi $ be given by (7.3) in \cite{Nirenberg}, following the first paragraph
in section 8.1 in \cite{Nirenberg}, we  obtain a unique solution $ Y \in \E^{2, \alpha}$
to \eqref{eq-ap-linear}, normalized to vanish at a fixed point on $ \mathbb{S}^2 $.
We denote such a $ Y $ by $ Y = \Phi (Z)$.

To prove estimates \eqref{eq-s57} and \eqref{eq-s58},  by the Remark
on page 365 in \cite{Nirenberg} and the proof following it,
we know it suffices to show
\be \label{eq-s86}
|| Y ||_{ C^{2, \alpha} } \leq C \lf[ || d \bar{\sigma}^2 ||_{C^{1, \alpha}}
 + || \frac{1}{\Delta} ( c_{1v }- c_{2u} ) ||_{C^\alpha} \ri],
\ee
where $ d \bar{\sigma}^2 = \gamma - ( d Z)^2 $, $c_{1v}$, $ c_{2u}$
are derivatives of $ c_1$, $c_2$ with respect to $v$, $u $ respectively.
On the other hand, by section 9 in \cite{Nirenberg},  to prove \eqref{eq-s86},
it suffices to establish an $C^{1, \alpha}$ estimate of $ \phi$:
\be \label{eq-s94}
|| \phi ||_{C^{1, \alpha} } \leq C || d \bar{\sigma}^2 ||_{C^{1, \alpha}} .
\ee
Therefore, in what follows, we will prove that there are positive numbers
$ \delta $ and $ C $, depending only on $ \sigma^0$,
such that  \eqref{eq-s94}  holds for any $ \sigma $ satisfying
$ || \sigma^0  - \sigma ||_{C^{2, \alpha}} < \delta $.

We first recall the fact that $ \phi $ is a solution to the second order elliptic equation
(6.15) in \cite{Nirenberg}.
For simplicity, we let
\be
L_\sigma ( \phi) = \mathcal{L}(\phi_u, \phi_v) , \ \
F_\sigma ( d \bar{\sigma}^2  ) = \mathcal{L}(c_1, c_2) - T,
\ee
where $  \mathcal{L}(\phi_u, \phi_v) $,
$ \mathcal{L}(c_1, c_2) - T$  are given as in  (6.16) and (6.14)  in \cite{Nirenberg}, then
(6.15) in \cite{Nirenberg}  becomes
\be \label{eq-ap-pde}
L _\sigma ( \phi ) + H_\sigma \phi  = F_\sigma ( d \bar{\sigma}^2 ),
\ee
where $ H_\sigma $ is the mean curvature of $ X^\sigma (\mathbb{S}^2 ) $ w.r.t $ X_3 $
(note that our $ H$ here equals $ 2 H $ in \cite{Nirenberg}).
On the other hand,  since  we have chosen $ \phi $  to be given by the integral
formula (7.3) in \cite{Nirenberg}, we know $ \phi $ is a special solution to \eqref{eq-ap-pde}
in the sense that $ \phi $ is $L^2$-perpendicular to the kernel of the operator
$ L_\sigma ( \cdot) + H_\sigma $ (See page 359 in \cite{Nirenberg}).
For any $ \sigma \in \M^{4, \alpha}$,
let $ Ker(\sigma) $ denote the space of solutions $ \psi $ to the
 the homogeneous equation
\be
L_\sigma ( \psi ) + H_\sigma \psi = 0 .
\ee
On page 360 in \cite{Nirenberg}, it was shown that $ Ker(\sigma)$
is spanned by the coordinate functions of $ X_3 $.

Note that the coefficient  of  \eqref{eq-ap-pde} depends only on the metric $ \sigma $.
Therefore, if $ \sigma $ is close to $ \sigma^0$  in $ {C^{2, \alpha}} $,
 we know  by Theorem 8.32 in \cite{Gilbarg-Trudinger} that, to prove
 the $C^{1, \alpha}$ estimate \eqref{eq-s94}, it suffices to prove the following
 $ C^0$ estimate
  \begin{equation}  \label{eq-ap-c0}
    ||\phi||_{C^0} \le C || d \bar{\sigma}^2 ||_{C^{1, \alpha} },
  \end{equation}
  where $ C $ is  some positive constant independent on $\sigma $, provided $ \sigma$ is
  sufficiently close to $ \sigma^0 $ in $ C^{2, \alpha}$.

Suppose \eqref{eq-ap-c0}  is not true, then there exists $ \{ \sigma_i \} \subset \M^{4, \alpha}$
which converges to  $ \sigma^0 $ in $ C^{2, \alpha}$,  $ \{ d \bar{\sigma}^2_i \} \subset \Sp^{1, \alpha}$
with $ || d \bar{\sigma}^2_i ||_{C^{1, \alpha} } = 1 $, and a sequence of numbers $ \{ C_i \}$ approaching
$ + \infty$ so that the corresponding $ \phi_i $ (of $ Y = Y_i $) satisfies
$$ || \phi_i ||_{C^0} \geq C_i . $$
Consider $ \xi_i =  \phi_i / || \phi_i ||_{C^0}$, then $ \xi_i $ satisfies
\be \label{eq-ap-xi}
L_{\sigma_i} ( \xi_i ) + 2 H_{\sigma_i} \xi_i = || \phi_i ||^{-1}_{C^0} F_{\sigma_i} ( d \bar{\sigma}_i^2) .
\ee
By Theorem 8.32 in \cite{Gilbarg-Trudinger}, we conclude from \eqref{eq-ap-xi} and
the facts $ \{ \sigma_i \} $ converges to  $ \sigma^0 $ in $ C^{2, \alpha}$
and $ || \xi_i ||_{C^0} = 1 $ that
\be \label{eq-ap-xi2}
|| \xi_i ||_{C^{1, \alpha} } \leq C,
\ee
where $ C $ is some positive constant independent on $ i $.
Now \eqref{eq-ap-xi2} implies that $\xi_i$ converges in $C^1$ to some $\xi$ which is also in
$C^{1,\alpha}$. Moreover, $ || \xi ||_{C^0} = 1 $.
By \eqref{eq-ap-xi}, $\xi$ is a weak solution to the equation
\be \label{eq-ap-xi3}
L_{\sigma^0} \xi +H_{\sigma^0} \xi=0  .
\ee
Since  $ \sigma^0 \in C^{5, \alpha}$, the coefficients of \eqref{eq-ap-xi3}
(given by (6.16) in \cite{Nirenberg}) are then in $ C^{3, \alpha}$, hence in $C^{2,1}$.
By Theorem 8.10 in \cite{Gilbarg-Trudinger}, we know $ \xi \in W^{4, 2} $, hence in $ C^2$.
Therefore, $ \xi $ is  a classic solution to \eqref{eq-ap-xi3}, i.e. $ \xi \in Ker(\sigma^0) $.
On the other hand, we know $ \phi_i $, hence $ \xi_i$,  is $ L^2$-perpendicular to $ Ker(\sigma_i )$ for each
$i $. Since $ \{ \sigma_i \}$ converges to $ \sigma^0$ in $ C^{2, \alpha}$ and $ \{ \xi_i \}$
converges to $ \xi $ in $ C^1 $, we conclude that $ \xi $ must be $ L^2$-perpendicular to
$Ker(\sigma^0)$. Hence, $ \xi $ must be zero. This is a contradiction to the fact $ || \xi ||_{C^0} = 1 $.
Therefore, we conclude that \eqref{eq-ap-c0} holds.

As mentioned earlier, once we establish the $C^0$  estimate  \eqref{eq-ap-c0},  we will  have
the $C^{1, \alpha}$ estimate \eqref{eq-s94}. Then we can proceed as in the rest of section 9 in
\cite{Nirenberg} to prove \eqref{eq-s86}, hence prove \eqref{eq-s57} and \eqref{eq-s58}.
\end{proof}

We note that the constants $ \epsilon $ and $ C $ in Lemma \ref{lma-ap-1} indeed can be chosen
as  $ \epsilon = \frac{1}{ 4 \bar{K}^2} $ and $ C = 2 \bar{K} $,
where $ \bar{K} $ is the constant in Theorem 2' on page 352 in \cite{Nirenberg}.
Therefore, by applying the exactly same iteration argument as on page 352-353 in \cite{Nirenberg},
one concludes from
Lemma \ref{lma-ap-2} that

\begin{lma} \label{lma-ap-3}
Let  $ \sigma^0 \in \M^{5, \alpha}$ be a metric with positive
Gaussian curvature. There exists positive numbers $ \delta$, $\epsilon $
and $ C$, depending only on $ \sigma^0$,
 such that for any  $ \sigma \in \M^{4, \alpha}$ satisfying
$$ || \sigma^0 - \sigma ||_{C^{2, \alpha}} < \delta ,$$
 if $ \tau \in \M^{2, \alpha}$ satisfying
$$ || \sigma - \tau ||_{C^{2, \alpha}} < \epsilon , $$
then there is an isometric embedding $ Y \in \E^{2, \alpha}$ of $(\mathbb{S}^2, \tau)$
in $ \R^3 $ such that
$$ || X - Y ||_{C^{2, \alpha} } \leq C || \sigma - \tau ||_{C^{2, \alpha}} .$$
Here $ X \in \E^{4, \alpha} $ is any given isometric embedding of $ (\mathbb{S}^2, \sigma)$.

\end{lma}

\vspace{.2cm}

\bibliographystyle{amsplain}

\begin{thebibliography}{10}

\bibitem{ADM61} Arnowitt, R., Deser, S. and Misner, C. W.,
{\sl Coordinate invariance and energy expressions in
general relativity}, Phys. Rev. ( 2) \textbf{122}, (1961),
997--1006.


\bibitem{Bartnik09} Bartnik, R., {\sl private communications}.

\bibitem{Bartnik88} Bartnik, R., {\sl Regularity of variational maximal surfaces}, Acta Math. \textbf{ 161} (1988), no. 3-4, 145--181.

\bibitem{BartnikSimon82}Bartnik, R. and Simon, L.{\it Spacelike hypersurfaces with prescribed boundary values and mean curvature}
Commun. Math. Phys. 87, 131-152(1982).



\bibitem{BYmass1}
Brown, J. David and York, Jr., James W.
\newblock Quasilocal energy in general relativity.
\newblock In {\em Mathematical aspects of classical field theory (Seattle, WA,
  1991)}, volume 132 of {\em Contemp. Math.}, pages 129--142. Amer. Math. Soc.,
  Providence, RI, 1992.

\bibitem{BYmass2}
Brown, J. David and York, Jr., James W.
\newblock Quasilocal energy and conserved charges derived from the           gravitational action.
\newblock{\em Phys. Rev. D (3)}, 47(4):1407--1419,1993.


\bibitem{Corvino00}
Corvino, J. , {\it Scalar curvature deformation and a gluing construction for the Einstein constraint
equations}, Commun. Math. Phys \textbf{214} (1) (2000), 137-189.



\bibitem{FanShiTam07}  Fan, X.-Q., Shi, Y.-G. and  Tam, L.-F.,
{\sl Large-sphere and small-sphere limits of the Brown-York mass},
  Comm. Anal. Geom., \textbf{17} (2009), 37–72.

\bibitem{F}F. J. Flatherty.{\it The boundary value problem for maximal hypersurfaces }Proc. Natl. Acad. Sci.
USA. Vol. 76, No. 10, pp. 4765-4767, October 1979
\bibitem{Gilbarg-Trudinger} Gilbarg, D. and  Trudinger, N. S.,
{\sl Elliptic partial differential equations of second order},
second edition,   Springer-Verlag, (1983).
\bibitem{HI}
Huisken, G. and Ilmanen, T, {\it The invers mean curvature flow and
the Riemannian Penrose Inequality}, J. Differential Geom.
\textbf{59} (2001), 353--437.
\bibitem{Kobayashi} Kobayashi, O., {\sl   A differential equation arising from scalar curvature function},  J. Math. Soc. Japan \textbf{34} (1982), no. 4, 665--675.
\bibitem{LY1} Liu, C.-C.M. and Yau, S.-T., {\it Positivity of quasilocal
mass}, Phys. Rev.Lett. 90(2003) No. 23, 231102

\bibitem{LY2} Liu, C.-C.M. and Yau, S.-T., {\it Positivity of quasilocal
mass II}, J. Amer.Math.Soc. 19(2006) No. 1.181-204.
\bibitem{Miao02}
Miao, P., {\sl Positive mass theorem on manifolds admitting corners
along a hypersurface},  Adv. Theor.
Math. Phys. \textbf{6} (2002), no. 6, 1163--1182 (2003).
\bibitem{MiaoTam08} Miao, P. and Tam, L.-F., {\sl On the volume functional of compact manifolds with boundary with constant scalar curvature}, to appear in  Calc. Var. Partial Differential Equations, arXiv: 0807.2693
\bibitem{Nirenberg} Nirenberg, L.,
{\sl The Weyl and Minkowski problems in differential geoemtry in the large},
Comm. Pure Appl. Math. \textbf{6} (1953), 337-394.
\bibitem{MST} N.  \'{O}. Murchadha, L.B. Szabados, and K.P. Tod,
{\it  Comment on " Positivity of quasilocal mass"} Phys. Rev.
Letter92(2004), 259001.



\bibitem{ShiTam02}
Shi, Y.-G. and  Tam, L.-F., {\it Positive mass theorem
 and the boundary
behaviors of compact manifolds with nonnegative scalar curvature},
J. Differential Geom. \textbf{62} (2002), 79--125.
\bibitem{ShiTam07}
Shi, Y.-G. and  Tam, L.-F., {\sl Rigidity of compact manifolds and
positivity of quasi-local mass},  Classical Quantum Gravity
\textbf{24} (2007), no. 9, 2357--2366.









\bibitem{ChristodoulouYau1986} Christodoulou, D. and Yau, S.-T., {\sl Some remarks on the quasi-local mass}  Mathematics and general relativity (Santa Cruz, CA, 1986),  Contemp. Math. \textbf{ 71} (1986),  Amer. Math. Soc., p. 9--14.




\bibitem{WangYau07} Wang, M.-T. and  Yau, S.-T., {\sl
A generalization of Liu-Yau's quasi-local mass}
Comm. Anal. Geom. 15 (2007), no. 2, 249--282.


\bibitem{WangYau08} Wang, M. -T. and Yau, S. -T., {\sl Isometric embeddings into
the Minkowski space and new quasi-local mass}, Comm. Math. Phys. \textbf{288}(3):919--942, 2009,
arXiv: 0805.1370v3












\end{thebibliography}

\end{document}